\documentclass[11pt,a4paper,reqno]{amsart}
\usepackage{amsfonts}
\usepackage{amsthm}
\usepackage[normalem]{ulem}
\usepackage{amsmath,amsfonts,amssymb,amsthm,accents}
\usepackage{mathtools}

\usepackage{mathrsfs}
\usepackage{array} 
\usepackage{comment}
\usepackage{indentfirst}
\usepackage{setspace}
\usepackage{amscd}
\usepackage[latin2]{inputenc}
\usepackage{bbm}

\usepackage{enumerate}
\usepackage[mathscr]{eucal}
\usepackage{enumitem}

\usepackage[margin=1in]{geometry}
\usepackage[bookmarks=false]{hyperref}
\hypersetup{
	colorlinks=true,
	linkcolor=blue,
	filecolor=blue,
	citecolor = blue,      
	urlcolor=cyan,
}

\theoremstyle{plain}
\newtheorem{Th}{Theorem}[section]

\newtheorem{Claim}[Th]{Claim}

\theoremstyle{definition}

\newtheorem*{Notation*}{Notation}
\newtheorem*{Rem*}{Remark}

\newtheorem{Prop}[Th]{Proposition}
\newtheorem{Theorem}{Theorem}

\theoremstyle{definition}
\newtheorem{Lemma}[Th]{Lemma}
\newtheorem{Cor}[Th]{Corollary}

\newtheorem{Question}{Question}
\newtheorem{Problem}[Question]{Problem}
\newtheorem{Conj}[Question]{Conjecture}

\theoremstyle{definition}

\newcolumntype{x}[1]{%
	>{\centering\hspace{0pt}}m{#1}}%

\newcommand{\tmu}{\nu}

\newcommand{\ra}{\rightarrow}

\newcommand{\cN}{\mathcal{N}}

\newcommand{\e}{\mathbb{E}}
\newcommand{\E}{\mathbb{E}}

\newcommand{\p}{\mathbb{P}}
\renewcommand{\P}{\mathbb{P}}
\newcommand{\muac}{\mu_{\text{ac}}}
\newcommand{\re}{{\mathbb{R}}}

\newcommand{\Z}{\mathbb{Z}}
\newcommand{\R}{\mathbb{R}}

\newcommand{\comp}{\mathbb{C}}
\newcommand{\lb}{\left(}
\newcommand{\rb}{\right)}
\newcommand{\nat}{\mathbb{N}}
\newcommand{\ep}{\varepsilon}

\newcommand{\uhp}{\mathbb{U}}
\newcommand{\lm}{\lambda}
\newcommand{\cF}{\mathcal{F}}
\newcommand{\Lsym}{\mathcal{L}^2_{\text{symm}}(\mu)}
\newcommand{\cart}{\mathcal{C}}

\newcommand{\mup}{\mu_{+}}
\newcommand{\mum}{\mu_{-}}

\DeclareMathOperator{\sgn}{sgn}

\DeclareMathOperator{\sprt}{\text{sprt}}
\newcommand{\U}{\mathbb{U}}

\begin{document}
	\onehalfspacing
	\title[Sharp transition]
 {A sharp transition in zero overcrowding and undercrowding probabilities for Stationary Gaussian Processes}

\author[Naomi D. Feldheim]{Naomi Dvora Feldheim}
 \address{Naomi Dvora Feldheim \hfill\break
Department of Mathematics, Bar-Ilan University}
\email{naomi.feldheim@biu.ac.il}

\author[Ohad N. Feldheim]{Ohad Noy Feldheim}
 \address{Ohad Noy Feldheim\hfill\break
Einstein Institute of Mathematics, Hebrew Univesity of Jerusalem}
\email{ohad.feldheim@mail.huji.ac.il}

    \author[Lakshmi Priya]{Lakshmi Priya M.E.}
    \address{Lakshmi Priya M.E.\hfill\break School of Mathematical Sciences, Tel-Aviv University}
    \email{lpriyame@gmail.com}
	
\thanks{N.F. and O.F. acknowledge the support of ISF grant 1327/19. L.P.M.E. acknowledges the support of ERC Advanced
Grant 692616, ERC consolidator grant 101001124 (UniversalMap) and ISF grant 1294/19.}
	\begin{abstract}
		{We study the probability that a real stationary Gaussian process has at least $\eta T$ zeros in $[0,T]$ (overcrowding), or at most this number (undercrowding). We show that if the spectral measure of the process is supported on $\pm[B,A]$, overcrowding probability transitions from exponential decay to Gaussian decay at $\eta=\tfrac{A}{\pi}$, while undercrowding probability undergoes the reverse transition at $\eta=\tfrac{B}{\pi}$.}
	\end{abstract}
	\maketitle
	
\section{Introduction}

Let $F:\re\to\re$ be a continuous centered stationary Gaussian process (SGP), that is, a shift-invariant random function whose finite marginal distributions are multi-normal with zero mean. Such a process is characterized by its \emph{covariance kernel}, $k(t) = \E[F(0) F(t)]$,
or, equivalently, by its \emph{spectral measure} $\mu$, which is the finite, non-negative, symmetric measure obtained as the inverse Fourier transform of $k$, so that,
\[
k(t) = \cF[\mu](t)= \int_\re e^{-i \lm t} d\mu(\lm).
\]
Throughout we assume $k(0)=1$, so that $\mu$ is normalized to be a probability measure.
The counting measure for the zero point process associated with $F$ is generated by 
\[
N_F(T) = N(T) =\# \{ t\in [0,T]: \: f(t)=0\}.
\]
The zeros of stationary Gaussian functions form a family of well-studied stochastic point processes with numerous applications~\cite{AT07, AW, CL67, kratzsurvey}. 
The expectation of $N_F(T)$ is given by the Kac-Rice formula~\cite{kac,rice},
\begin{equation*}
\E [N_F(T)] = \frac{\gamma}{\pi} T,\quad \text{where}\quad
\gamma^2 = -k''(0) = \int_\re \lm^2 d\mu(\lm).
\end{equation*}
Variance and higher moments of $N(T)$ have also been studied \cite{ABF21, Slud94}, along with results about clustering and limit theorems~\cite{AL21, KL01}.
Nevertheless, rare events involving such processes are generally not well-understood (see discussion in~\cite{bdfz} and~\cite[Q. 7]{ms}). 

Here we study linear deviations of $N_F(T)$ as $T$ tends to infinity. These consist of \emph{$\eta$-overcrowding} events, given by 
$\{N_F(T) \ge \eta T\}$ for $\eta > \e[N_F(1)]$, and 
\emph{$\eta$-undercrowding} events, given by
$\{N_F(T) \le  \eta T\}$ for $\eta < \e[N_F(1)]$.

Basu, Dembo, Zeitouni and the first author showed in~\cite{bdfz} that, for well-behaved SGPs and any $\eta>\E[N_F(1)]$, the probability of $\eta$-overcrowding decays at most exponentially with $T$. In contrast, the third author showed in \cite{lp2} that when the spectral measure is compactly supported and $\eta$ is sufficiently large, the $\eta$-overcrowding probability decays in a Gaussian fashion or faster. Here we establish a sharp phase transition between these two decay profiles, occurring at $\eta=\tfrac{A}{\pi}$ where $A$ is the supremum of the support of the spectrum. 


\begin{Theorem}\label{thm_overcrowding} 
  Suppose that $\mu$ is compactly supported with $\muac \not\equiv 0$. 
  Let $A$ be the smallest positive number such that $\text{sprt}(\mu) \subseteq [-A,A]$. Then: 
		\begin{enumerate}
			\item\label{item: over 1} There exist $c,d_1,d_2>0$ such that, for every $T\ge e^8$ and $\ep \in \left[d_1{{\tfrac {\log T}{\sqrt T}}}, d_2\right]$, we have
			\begin{align*}
			\p \lb N_{F}(T) \geq \frac{A}{\pi}  T + \ep T\rb \le \exp \lb -c (\ep/\log \ep)^4 T^2 \rb.
			\end{align*} 
			\item \label{item: over 2} For every $\ep >0$ there exists $c >0$ such that, for all $T\ge 1$, we have
			\begin{align*}
			\p\lb N_F(T) \geq \frac{A}{\pi} T - \ep T \rb \ge\exp(-cT). 
			\end{align*}
		\end{enumerate}
	\end{Theorem}
An analogous phase transition is established for undercrowding, in case $\mu$ has a spectral gap.

\begin{Theorem}\label{thm_undercrowding}  Suppose that $\mu$ is compactly supported with $\muac \not\equiv 0$. Let~$B$ be the largest and~$A$ the smallest non-negative numbers such that $\sprt(\mu) \subseteq [-A,-B]\cup [B,A]$.
Then:
\begin{enumerate}
	\item\label{item: under 1} There exist $c,d_1,d_2>0$ such that, for every $T\ge e^8$ and
 $\ep \in \left[d_1{\tfrac {\log T}{\sqrt T}}, d_2\right]$,
 we have
	\begin{align*}
	\p \lb N_F (T)  \leq \frac{B}{\pi}T - \ep T\rb \le \exp(-c(\ep/ \log \ep)^4 T^2).
\end{align*} 
	\item\label{item: under 2} For every $\ep >0$, there  exists $c >0$ such that, for all $T\ge 1$, we have
 \begin{align*}
	\p \lb N_F(T) \le \frac{B}{\pi}T + \ep T \rb \geq \exp(-cT).
	\end{align*}
	\end{enumerate}
	
	\end{Theorem}

The remainder of the paper is structured as follows. 
In Sections~\ref{sec: story low} and~\ref{sec: story up} we describe the main steps of the proofs of the lower and upper bounds, respectively. In Section~\ref{sec:tightness} we discuss tightness of the results. In Section~\ref{sec:2.1} we discuss their role in the broader context of the study of zeros of SGPs. In Section~\ref{sec:det ang} we mention deterministic analogues which inspired the proof. Section~\ref{sec:fut} is dedicated to open problems and future research.
Section~\ref{sec: prelim} is a preliminary section containing useful results about entire functions and Gaussian processes. In Section~\ref{sec:lower bound} we prove 
a result (Proposition~\ref{prop:lb} below) which implies the lower bounds in both theorems. Section~\ref{sec:over} contains the proof of the upper bound of Theorem~\ref{thm_overcrowding}, and Section~\ref{sec:under} contains a reduction of the upper bound of Theorem~\ref{thm_undercrowding} to that of Theorem~\ref{thm_overcrowding}.

\subsection{Main steps of the proof: lower bound.}\label{sec: story low}
The lower bound in both theorems readily follows from the following proposition, which is concerned with obtaining tight control over~$\frac{N(T)}{T}$ and may be of independent interest.

\begin{Prop} \label{prop:lb} 
Suppose that $\int\lambda^{2}\log^2{\lambda}\, d\mu(\lambda)<\infty$ and $\mu([X-\ep, X+\ep])>0$ for some $X\in (0,\infty)$ and $\ep>0$.  Then there exists $c>0$ such that
\begin{equation}\label{eq:eq lb}
\P\left(\frac{N(T)}{T} \in \left[\frac{X-2\ep}{\pi},\frac{X+2\ep}{\pi}\right]\right)\ge \exp(-cT),\quad
\text{for sufficiently large } T.
\end{equation}
Moreover, if $\lim\sup_{\delta\to 0}\frac{\mu([X-\delta, X+\delta])}{2\delta}=\infty$, then \eqref{eq:eq lb} holds for any choice of $c>0$.
\end{Prop}
This is obtained by estimating the probability that $F$ closely imitates a pure wave.
The lower bounds in Theorems~\ref{thm_overcrowding} and~\ref{thm_undercrowding} are then obtained by the choice of $X=A$ and $X=B$ respectively. 

\subsection{Main steps of the proof: upper bound.}\label{sec: story up}
The proof of the upper bound given in Theorem~\ref{thm_overcrowding} uses classical tools from complex analysis. In fact, rather than providing a direct upper bound for the number of real zeros $F$ has in $[0,T]$, we shall provide a bound for the number of complex zeros of the analytic extension of $F$ in a neighbourhood of this interval. A similar idea was used in~\cite{bdfz}.
 

An entire function $f: \comp \ra \comp$ is  of \emph{exponential type}, if there exist $\sigma,C>0$ such that 
\begin{equation}\label{eq:exp_type}
|f(z)| \le C\exp(\sigma|z|),\quad\forall z \in \comp.
\end{equation}
The \emph{exponential type of $f$} is then defined as the infimum over all $\sigma$ for which this inequality is satisfied for some value of $C$.
An entire function $f: \comp \ra \comp$ is said to belong to the \emph{Cartwright's class} $\mathcal{C}$ if the following two conditions hold.
		\begin{enumerate}
			\item $f$ is of exponential type $\sigma$ for some $\sigma\in [0,\infty)$.
			\item $f$ satisfies the integral condition:
		\begin{align}\label{eq:defn_cartwright}
			\int_{\re} \frac{\log^{+}|f(t)|}{1+t^2} dt < \infty.
			\end{align}
		\end{enumerate}
   Zeros of functions from Cartwright's class have special properties, see \cite[\S 16.1]{levin} and \cite{zeros_cart} for more details.

The first step in our proof is to show that $F$ is almost surely in Cartwright's class, and of exponential type at most $A$ (Proposition~\ref{new_lemma} below). This guarantees a tight upper bound on $\log |F|$ (Corollary~\ref{cor:analogue_of_pl}).
The celebrated Jensen's formula (Theorem~\ref{Jensens} below) is then used in order to relate weighted zero counts of $F$ to evaluation and integration of $\log|F|$ around a point. To establish the upper bound we further average this formula over a carefully selected set of points, so that the total weight given to each real zero in $[0,T]$ will be roughly the same. The challenge then lies in providing probabilistic bounds on $\log |F|$ and on the event of existence of suitable points.

The upper bound of Theorem~\ref{thm_undercrowding} is obtained by applying a non-trivial coupling between $F$ and another SGP $G$ (Claim~\ref{claim_undercrowding} below), under which undercrowding for $F$ corresponds to overcrowding for $G$. Applying Theorem~\ref{thm_overcrowding} to $G$ completes the argument. This useful coupling was suggested by Eremenko and Novikov~\cite{EN04}. 

\subsection{Tightness of the bounds}\label{sec:tightness}

The one-sided bounds of Theorems~\ref{thm_overcrowding} and~\ref{thm_undercrowding} often capture the exact probability decay above and below the transition point, up to a constant in the exponent. Here we survey results concerning the conditions for this to hold.

\textbf{Exponential upper bounds.}
Results of Basu, Dembo, Zeitouni and the first author~\cite{bdfz} imply that for SGPs with compactly supported spectral density and absolutely integrable covariance kernel, any (upper or lower) linear deviation of $N_F(T)$ from its mean indeed decays exponentially or faster\footnote{in fact, in \cite{bdfz} this is shown for a slightly larger family of SGPs.}. We conjecture that this property should hold whenever the spectral density is bounded, and has some finite exponential moment (see Conjectures~\ref{conj: exp conc} and~\ref{conj: tail} below).

\textbf{Gaussian lower bounds.} 
In \cite{FFJNN20}, Jay, Nazarov and Nitzan together with the first two authors, showed that the probability of persistence, namely $0$-undercrowding, is of probability at least $\exp(-C T^2)$ if the spectral measure is compactly supported, having a non-trivial absolutely continuous component.
By monotonicity, this implies 
a Gaussian lower bound on $\eta$-undercrowding probability for all $\eta$. 
The third author \cite{lp2} showed an analogous Gaussian lower bound on $\eta$-overcrowding,  for spectral measures with finite moments and non-trivial absolutely continuous component. Hence the first part of both Theorems~\ref{thm_overcrowding} and ~\ref{thm_undercrowding} is tight up to a constant in the exponent.

\subsection{Deterministic analogues}\label{sec:det ang}
The number of zeros of Cartwright's class functions in the complex plane inside a ball around the origin, is asymptotically restricted by their exponential growth. This is shown by the following theorem attributed to Levinson and Cartwright, which is analogous to Theorem~\ref{thm_overcrowding}.

	
\begin{Th}[{\cite[Ch. 17, Thm. 1]{levin}}]\label{res_asymp_zeros} Let $f \in \cart$ be  of exponential type $A>0$. Then
	\begin{align*}
	  \lim_{R \ra \infty}  \frac{n_{f}^{+}(R)}{R} = \lim_{R \ra \infty}  \frac{n_{f}^{-}(R)}{R} = \frac{A}{\pi},
	    \end{align*}
	    where $n_{f}^{\pm}(R)$  denotes the zero count of $f$ in $\{z \in \comp:|z| <R \text{ and } \text{Re}(z)\gtrless 0\}$.
\end{Th}

This has been used by Eremenko and Novikov to study the asymptotic zero density of functions with a spectral gap at the origin,
a result analogous to Theorem~\ref{thm_undercrowding}.
\begin{Th}[{\cite[Prop. 1]{EN04}}]\label{res:spectral_gap} Let $f: \re \ra \re$ be a continuous function whose spectrum is supported on $\pm[B,A]$, for some $A,B>0$, for which $\int_{\re} \frac{\log^{+}|f(t)|}{1+t^2} dt < \infty$.
Then
\begin{align*}
\liminf_{R \ra \infty} \frac{N_{f}(R)}{R} \geq \frac{B}{\pi}.
\end{align*}
	\end{Th}
N.b. that for Theorem~\ref{res:spectral_gap} to hold in such generality, one needs to generalise the notion of a Fourier transform, see \cite{EN04} for more details.

For an SGP $F$ whose spectral measure is supported on $\pm[B,A]$, these results tell us that, asymptotically, $\tfrac{N_F(T)}{T}\in [\frac{B}{\pi},\frac{A}{\pi}]$. The purpose of this paper is to provide a probabilistic estimate for the rate of convergence. While we do not use Theorems~\ref{res_asymp_zeros} and~\ref{res:spectral_gap} themselves, ideas from the proofs of these results will be used in proving Theorems \ref{thm_overcrowding} and \ref{thm_undercrowding} respectively. 




\subsection
{Background}
\label{sec:2.1}

For many years there has been little progress in studying large deviations of the zero count of SGPs.
It has been anticipated that, similar to many classical models for random sums, the probability of a large deviation in $N(T)$ should be asymptotically exponential in $T$.
Until recently, such exponential concentration was not established even for a single non-trivial example of SGPs (as
mentioned in lecture notes by Tsirelson~\cite{Tsirel} and Sodin~\cite{ms}); even though such results were obtained
for related models, 
including complex zeros of the planar and hyperbolic Gaussian analytic functions~\cite{Krishnapur06, ST05} and nodal component count of Gaussian Laplace eigenfunctions~\cite{NS09, lp1,Rozenshein}.
    Some exponential bounds on deviations were also obtained for zeros of high-degree Kostlan polynomials and other Gaussian models on manifolds~\cite{GW11}, as well as counts of nodal components and  other topological events in high dimensional smooth Gaussian fields with fast covariance decay ~\cite{BMR20}.

The extremely fast decay of $\eta$-overcrowding for sufficiently large values of $\eta$, has been pointed out in \cite{lp2}. There, larger than linear overcrowding has also been considered, and in the case of compactly supported spectral measures with absolutely continuous component, it was shown that $\log \p(N_F(T) \geq n) \asymp   -  n^2 \log\lb n/T \rb$ for $n\ge CT$, with some constant $C>0$.


The particular case of $0$-undercrowding, also known as the \emph{persistence event}, gained substantial attention by both physicists~\cite{BMS13} and mathematicians~\cite{AS15}. In 1962 Slepian introduced his famous comparison inequality in order to bound these events~\cite{Slepian62}, but this was not enough to determine the rate of decay for many examples of interest (see e.g.~\cite{DHZ96}). 
It was only in the last decade that major advances were made in understanding persistence and estimating its probability (e.g.~\cite{AMZ21,  DM15, DM17, KK16}). Using new spectral techniques, the decay rate of persistence probabilities was shown to depend mainly on the behaviour of the spectral measure near the origin~\cite{FFN21}, with asymptotically exponential behaviour corresponding to existence of a spectral density at the origin~\cite{FFM22}. The methods developed in these works play a major role in the proof of Proposition~\ref{prop:lb}.

\subsection{Future work}\label{sec:fut}
Here we survey several directions for future work concerning overcrowding and undercrowding events.

\textbf{Transition window.}
While the results of Theorems~\ref{thm_overcrowding} and~\ref{thm_undercrowding} demonstrate the existence of a phase transition in the overcrowding and undercrowding probabilities, they leave to be desired in understanding the fine nature of this transition and recovering the transition window. 
Theorem~\ref{thm_overcrowding} implies that for $\beta \in (\tfrac 1 2, 1]$ we have
    $$\p \lb N(T) \geq \frac{A}{\pi} T + T^{\beta}\rb\leq \exp({-T^{4\beta - 2-\ep}})$$
for any $\ep>0$, while Theorem~\ref{thm_undercrowding} implies an analogous bound for undercrowding.
For $\beta\in [\tfrac 1 2,\tfrac 3 4]$ this bound is worse than the bound obtained in \cite{bdfz}, and is unlikely to be tight. Our first question is concerned with improving this bound and finding a matching lower bound. 

\begin{Problem}
Obtain tight upper and lower bounds for $\p \lb N(T) \geq \frac{A}{\pi} T + T^{\beta}\rb$ for $\beta\in[0,1]$.
\end{Problem}

\textbf{Undercrowding for non-compactly supported spectrum.}
It seems natural that sub-Gaussian undercrowding probability decay should occur in all processes demonstrating a spectral gap, and that the condition of compact spectral support is not needed. This restriction in Theorem~\ref{thm_undercrowding} stems only from our use of Theorem~\ref{thm_overcrowding} in obtaining it.

\begin{Problem}
Remove the condition that $\mu$ is compactly supported in Theorem~\ref{thm_undercrowding}.
\end{Problem}

\smallskip

\textbf{General exponential upper bounds.}
Proposition~\ref{prop:lb} indicates that the exponential upper bounds of \cite{bdfz} do not hold for  spectral measures with unbounded spectral density, or a singular spectral component, away from $\pm \pi \E[N_F(1)]$. Exponential upper bounds also require the measure to have some exponential moment, as otherwise $N_F(T)$ cannot decay exponentially. We conjecture that these are the true obstructions for such a bound. 

\begin{Conj}\label{conj: exp conc}
If $\mu$ has bounded density, and some finite exponential moment, then for any $\alpha>0$ there exists $c>0$ such that \[\P(|N_F(T)-\E N_F(T)| \ge \alpha T) \le e^{-cT}.\]
\end{Conj}

In fact, we conjecture the following stronger statement. 
\begin{Conj}\label{conj: tail}
    If $a>\pi \E[N_F(1)]$ is such that $\mu|_{[a,\infty)}$ has bounded density and some finite exponential moment, then there exists $c>0$ such that $\P\left(N_F(T)>\frac{a}{\pi}T\right)\le e^{-cT}$. 
\end{Conj}

A similar property should hold for undercrowding.

\smallskip
\textbf{Exponents and conditional behaviour.}
Lastly, we state two additional problems concerned with exact overcrowding and undercrowding exponents and conditional behaviour. 
\begin{Problem}
    Describe conditions under which the limit $\displaystyle\lim_{T\to\infty} \tfrac 1 T\log \p \left( N_F(T) \ge \eta T\right)$ exists.
\end{Problem}
\begin{Problem}
Describe the process $F$ conditioned on the event $\{N_F(T)\ge \eta T\}$. In particular, describe the empirical power spectrum $\displaystyle\lim_{T\to\infty}    \frac{1}{T}\left|\mathcal F^{-1}\left[F\mathbbm{1}_{[0,T]}\right]\right|^2$, where $F$ is the conditioned process.
\end{Problem}
Analogous problems could be stated concerning undercrowding events. A possible approach to these problems may be to refine known large deviation principles for empirical measures of SGPs, established by Donsker-Varadhan~\cite{DV85} and Bryc-Dembo~\cite{BD95}, whose current conditions are too restrictive to be applied directly.





 \section{Preliminaries}\label{sec: prelim}
	
	\subsection{Theory of entire functions}\label{subs: perlim: entire}

 We require a few classical results about entire functions. 
For $z\in \comp$ and $r>0$, denote by $B(z,r) :=\{w\in\comp: \ |w-z|<r\}\subseteq \comp$ the disc of radius $r$ around $z$, and by
$n_{f}(z,r)$ the number of zeros, counted with multiplicity, of an entire function $f: \comp \ra \comp$
in that disc. 
The following is the famous zero-counting formula by Jensen. 


		\begin{Th}[Jensen~{\cite[\S 2.3]{levin}}] \label{Jensens} Let $f: \mathbb{C} \ra \comp$ be an entire function such that $f(0) \neq 0$. 
  Suppose that $f$ has no zeros on $\partial B(0,R)$. Then
		\begin{align*}
		\int_{0}^{R} \frac{n_f(0,t)}{t} dt = \frac{1}{2\pi}\int_{0}^{2\pi} \log|f(Re^{i\theta})| d\theta - \log|f(0)|.
		\end{align*}
	\end{Th}

We shall also require the Phragm\'en--Lindel\"of theorem.
Let $\U:= \{z\in \comp: \text{Im}(z)> 0\}$ denote the upper half plane, and $\overline\U$ denote its closure. The notion of functions of exponential type $\sigma$ naturally extends to holomorphic functions on $\mathbb{U}$.


	\begin{Th}[Phragm\'en--Lindel\"of~{\cite[\S 6.1, Theorem 3]{levin}}]  \label{pl} Let $f:\overline{\uhp} \ra \comp$ be a continuous function, holomorphic on  $\U$, of exponential type $\sigma$, which satisfies $|f| \leq M$ on $\re$.
  Then for every $z = x+iy \in \overline{\mathbb{U}}$, 
	\begin{align*}
	|f(x+iy)| \leq M e^{\sigma y}.
	\end{align*}
\end{Th}

From this we draw the following corollary, concerning functions in the Cartwright's class, which rather than being bounded, satisfies the slow growth assumption \eqref{eq:defn_cartwright}. This result is obtained in the course of the proof of {\cite[\S 16.1, Theorem 1]{levin}} and brought here for completeness.

\begin{Cor}\label{cor:analogue_of_pl} Let $f$ be a Cartwright's class function which is of exponential type $\sigma >0$. Then for every $z=x+iy \in \comp \setminus \re$, we have
\begin{align*}
\log |f(z)| \leq \frac{|y|}{\pi} \int_{\re} \frac{\log^{+}|f(t)|}{|t-z|^2} dt + \sigma |y|. 
\end{align*}
	\end{Cor}
\begin{proof}
Assume, without loss of generality that $z\in \uhp$. Denote \[g(z)=\frac{y}{\pi}\int_{\re} \frac{\log^{+}|f(t)|}{|t-z|^2} dt,\]
which converges by \eqref{eq:defn_cartwright}. This is a harmonic function on $\U$ with boundary values on $\re$ coinciding with $\log^+|f|$. Denote the harmonic conjugate of $g$ by $\tilde g$ and write
\[F(z)=f(z)e^{-g(z)-i\tilde g(z)}.\]
Since $g\geq 0$, we have $|F(z)| \leq |f(z)| \exp(-g(z)) \leq |f(z)|$, so that $F$ is analytic of exponential type at most $\sigma$ on $\U$. Moreover
$|F(t)|\le 1$ for all $t\in \re$, so that we may apply Theorem~\ref{pl} to obtain $|F(x+iy)| \leq \exp(\sigma y)$. Taking logarithm on both sides, the corollary follows. 
\end{proof}
\subsection{Tools from the theory of SGPs}\label{subs: perlim: gaussian}

Throughout this section, let $F$ be a fixed centred SGP with spectral measure $\mu$, the moments of which are all finite.
It follows that $F$ is smooth and its higher derivatives $F^{(n)}$ are centred SGPs with spectral measure $\mu_n$ given by $d\mu_n(x) = x^{2n} d\mu(x)$, and covariance kernel given by
\begin{align}\label{eq:der_of_SGP}
\e[ F^{(n)}(x)\cdot   F^{(n)}(0)] = (-1)^{n} ~  k^{(2n)} (x),\text{ for $x\in \re$,}
\end{align}
where $k$ is the covariance kernel of $F$.
Define the space of symmetric square integrable functions with respect to $\mu$ by
\[
\Lsym := \left\{ \varphi: \re\to\comp: \:\: \int_\re |\varphi|^2 d\mu <\infty, \:\: \varphi(-\lambda)=\overline{\varphi(\lambda)} \:\:\forall \lm\in \re\right\}.
\]


$F$ boasts the following orthonormal representation.
\begin{Lemma}
    [{\cite[Lem. 3.7]{FFN21}}] \label{lem:series_repn} 
Let $\{\varphi_n\}$ be an orthonormal basis of $\Lsym$ and write $\psi_n(t) = \int_\re e^{-i\lm t}\varphi_n(\lm) d\mu(\lm)$.
 Then
\[
F \overset{d}{=} \sum_n \zeta_n \psi_n, \quad \zeta_n \overset{\text{i.i.d.}}{\sim} \mathcal{N}_\re(0,1).
\]
\end{Lemma}


The following lemma shows that $F$ grows at most linearly on the real line, providing a probabilistic bound on its growth rate.

\begin{Lemma}\label{lem:lineargrowth} 
There exists $c>0$ such that for every $M>0$ sufficiently large 
\[\P\Big(\forall s\in \re\,:\,|F(s)| \le  M(1+|s|)\Big)\ge 1-e^{-cM^2}.\]   
\end{Lemma}

	\begin{proof}[Proof of Lemma~\ref{lem:lineargrowth}]
        By a simple application of Dudley's bound~\cite[Thm. 1.3.3]{AT07} and the Borell-TIS inequality~\cite[Thm. 2.1.1]{AT07} (see~\cite[Lem. 3.10]{lp2} for a full proof), there exist $a,c  \in (0,1)$ such that for every  $M$ sufficiently large  we have
	\begin{align*} 
	\p\left(\sup_{[0,aM]} |F| \geq M\right) \leq \exp(-cM^2).
	\end{align*}
        By the stationarity of $F$, the events
        \[{E}_{k,M} := \left\{\sup\{ |F(t)| : t\in [akM,a(k+1)M]\}>\frac{(|k|+1)M}{2} \right\}\]
        satisfy 
		$\p({E}_{k,M}) \leq \exp(-ck^2M^2/4)$. Taking a union bound, we obtain, for sufficiently large $M$,
        \[\p(\cup_{k \in \Z} {E}_{k,M}) \le \exp(-cM^2/5).\] 
        Observing that
        $\{\exists s \in \re\ :\ |F(s)| >  M(1+|s|)\}\subset \cup_{k \in \Z} {E}_{k,M},$ the lemma follows.
    \end{proof}

We also need bounds on \emph{ball event} probabilities.

\begin{Lemma}[Application of {\cite[Lemma 3.4]{lp2}}] \label{lem:smallballprob} 
Suppose that $\muac \neq 0$. Then there exist $\alpha,\beta>0$ such that for all $T\ge 1$ and $\ell>0$ we have:
\[
\P\left( \sup_{[0,T]} |F| \le \ell \right) \le e^{\beta T^2} \ell^{\alpha T}.
\]
\end{Lemma}


\begin{Lemma}[{\cite[Lemma 3.17]{FFN21}}]\label{lem: FFN ball}
Suppose that $\int |\lm|^\delta d\mu(\lm)<\infty$ for some $\delta>0$. Then there exist $\ell_0>0$ and $c>1$ such that
for all $T\ge 1$ and $\ell>\ell_0$ we have:
\[
\p\left(\sup_{[0,T]} |F| \leq \ell \right) \ge 
\p\left( c |F(0)|\le \ell \right)^T.
\]
\end{Lemma}

The following famous comparison inequality is due to Anderson.

\begin{Lemma}[{\cite{Anderson55}}]\label{lem:ander}
Let $X$ and $Y$ be independent centered Gaussian processes on $I$. Then for any $\ell>0$,
\[
\p \left( \sup_I |X+Y|\le \ell \right) \le \p \left(\sup_I |X| \le \ell \right).
\]
\end{Lemma}

We shall also need the celebrated Gaussian correlation inequality.
\begin{Lemma}[{\cite{LM17,Royen14}}]\label{lem:GCI}
Let $X$ be a centred Gaussian vector in $\R^d$. Then, for any convex sets $K, L\subset \R^d$, symmetric around $0$, we have 
\[
\P(X\in (K\cap L)) \ge \P(X\in K) \,\P(X\in L).
\]
\end{Lemma}

Lastly, we state standard bounds on the tail of the Gaussian distribution.
\begin{Lemma}[{\cite[Chapter 1.2]{AT07}}]\label{lem: tail}
 Let $Z\sim \mathcal{N}_\re (0,1)$. Then for any $x>0$:
 \[\frac 1{\sqrt{2\pi}}\left(\frac 1 x - \frac 1{x^3}\right) e^{-x^2/2} \le  \p(Z>x) \le \frac 1{\sqrt{2\pi}}\frac 1 x  e^{-x^2/2}.\]
In particular, for $x\ge 2$ it holds that $e^{-x^2} \le \p(Z> x)\le e^{-x^2/2}$.
\end{Lemma}

\section{Lower bounds}\label{sec:lower bound}
This section is dedicated to the proof of Proposition~\ref{prop:lb}.
Assume, without loss of generality,  that $X=1$, $0<\ep<\frac 1 {10}$ and $T>1$.  Observe that there must exist $a\in (1-\ep,1+\ep)$ such that   
		\begin{align} \label{eq:mu I}
		    \mu\left(\left[a-\tfrac{1}{10 T},a+\tfrac{1}{10 T}\right]\right) \geq \tfrac{1}{ 10 \ep T}\, \mu([1-\ep,1+\ep]).
		\end{align}
  Fix such $a$ and denote $$I=\left[a-\tfrac{1}{10 T},a+\tfrac{1}{ 10 T}\right] \text{ and } \pm I=I\cup -I.$$
Take an orthonormal basis $\{f_n\}_{n=1}^{\infty}$ for $\Lsym$, with
$f_1 = (\mu\left(\pm I\right))^{1/2} \mathbbm{1}_{\pm I}.$
From Lemma \ref{lem:series_repn} we obtain that $$F(x) = \sum_{n=1}^{\infty} \xi_n \widehat{f}_n (x), \text{ where } \xi_n \overset{\text{i.i.d.}}{\sim} \cN(0,1).$$
Denote $G:=F-\xi_1 \widehat{f_1}$ and observe that $G$ is independent of $\xi_1$.
We compute,
\begin{align*}
    \sqrt{2\mu (I)} \widehat{f_1}(x) 
    &=2\int_{I}\cos(\lm x)d\mu(\lm) =2 \mu(I) \cos(ax)+  2\int_{I} \left(\cos(\lm x)-\cos(a x)\right) d\mu(\lm), \\
    -\sqrt{2\mu (I)} \widehat{f_1}'(x) 
    &=2\int_{I}\lambda\sin(\lm x)d\mu(\lm) =2 a \mu(I) \sin(ax)+  2\int_{I} \left(\lambda\sin(\lm x)-a\sin(ax)\right) d\mu(\lm).
\end{align*}
Since both $\sin$ and $\cos$ are 1-Lipschitz, we conclude that for any $|x|\le T$, 
\begin{align*}\label{eq: close to cos}
\left| \frac{\widehat{f_1}(x)}{\sqrt{2\mu(I)}}  -\cos(ax)\right|&\le
\frac 1{\mu(I)} \int_I |\lm-a||x|  d\mu(\lm)  < \frac{1}{10}, \quad\text{and}\\
\left| \frac{\widehat{f_1}'(x)}{a\sqrt{2\mu(I)}} + \sin(ax)\right|&\le
\frac 1{a \mu(I)} \int_I |\lm-a|(|a x |+1)  d\mu(\lm)  \le \frac 1 {10} \left(1 + \frac{2}{1-\ep}\right) <\frac{4}{10},
\end{align*}
where we used the fact that $|\lm-a|\le \frac{1}{10T}$ for all $\lm\in I$.
Observe that $G$ is almost surely continuously differentiable and that $G'$ is a centred Gaussian process, by our assumption  $\int\lambda^{2}\log^2{\lambda} d\mu(\lambda)<\infty$.
Given $L>0$, define the events
\[
E_1 = \left\{ \sqrt{2\mu(I)}\xi_1  \geq 10L\right\}, \quad E_2=\left\{\sup_{[0,T]} |G|\leq L\right\}, \quad E_3=\left\{\sup_{[0,T]} |G'|\leq aL\right\},
\]
and observe that $E_1$ is independent from $E_2$ and $E_3$.
We deduce that on the event $E_1\cap E_2\cap E_3$, we have
\begin{align*}
\left|\frac{F(x)}{\sqrt{2\mu(I)}\xi_1}- \cos(ax)\right|&=
\left|\frac{\widehat f_1(x)\xi_1}{\sqrt{2\mu(I)}\xi_1}+
\frac{G(x)}{\sqrt{2\mu(I)}\xi_1}-\cos(ax)\right|\\
&\le 
\left|\frac{\widehat f_1(x)}{\sqrt{2\mu(I)}}- \cos(ax)\right|+\left|\frac{G(x)}{\sqrt{2\mu(I)}\xi_1}\right| <  \frac12.\\ 
\left|\frac{F'(x)}{a\sqrt{2\mu(I)}\xi_1}+\sin(ax)\right|&=
\left|\frac{\widehat{f_1}'(x)\xi_1}{a\sqrt{2\mu(I)}\xi_1}+
\frac{G'(x)}{a\sqrt{2\mu(I)}\xi_1}+\sin(ax)\right|\\
& \leq
\left|\frac{\widehat{f_1}'(x)}{a\sqrt{2\mu(I)}}+ \sin(ax)\right|+\left|\frac{G'(x)}{a\sqrt{2\mu(I)}\xi_1}\right| < \frac{1}2.
\end{align*}
Denote $S_1= \left(\tfrac{\pi}{a}\Z+\left[-\tfrac{\pi}{4a},\tfrac{\pi}{4a}\right]\right)\cap [0,T]$ and 
$S_2= \left(\tfrac{\pi}{a}\Z+\left[\tfrac{\pi}{4a},\tfrac{3\pi}{4a}\right]\right)\cap [0,T]$, and observe that $S_1\cup S_2= [0,T]$.
We conclude that on the event $E_1\cap E_2\cap E_3$ we have
\begin{align}
\label{eq: S1}\forall x\in S_1&: \quad  \left|\frac{F(x)}{\sqrt{2\mu(I)}\xi_1}\right|>\cos(\tfrac \pi 4)-\tfrac1 2
\\
\label{eq: S2}\forall x\in S_2&: \quad \left|\frac{F'(x)}{a\sqrt{2\mu(I)}\xi_1}\right|>\sin(\tfrac \pi 4)-\tfrac{1}2 
\\
\label{eq: sgn} \forall k\in \Z\cap [0,T]&: \quad \sgn\left( F(\tfrac{k\pi}{a}) \right)=(-1)^{k}.
\end{align}
By \eqref{eq: S1}, $F$ has no zeros in $S_1$. By~\eqref{eq: sgn}, it has at least one zero in each interval of $S_2$ except the last, while by~\eqref{eq: S2}, it cannot have more than one zero in each interval of $S_2$. Since
there are $\lfloor\frac{aT}{\pi}\rfloor$ such intervals, and we deduce that
$\frac {N(T)}{T} \in [\frac{1-2\ep}{\pi},\frac{1+2\ep}{\pi}]$ for sufficiently large $T$. Therefore
\begin{equation}\label{eq: lower}
\p \left( \frac{N(T)}T \in \left[\frac{1-2\ep}{\pi} ,\frac{1+2\ep}{\pi} \right] \right) \ge \p(E_1\cap E_2\cap E_3) = \p (E_1) \p( E_2\cap E_3 ). 
\end{equation}
We are thus left with choosing $L$ and obtaining a bound on $\p (E_1)$ and $ \p( E_2\cap E_3 )$.

Writing $\kappa=\frac{ \ep}{10\mu([1-\ep,1+\ep])}$ and using~\eqref{eq:mu I} together with the Gaussian tail bound in Lemma~\ref{lem: tail}, we get, for sufficiently large $T$,
\[
\p(E_1)= \p \left(\xi_1\ge 5 \sqrt{\tfrac{2}{\mu(I)}} L\right) \ge 
\p\left(\xi_1 \ge 5\sqrt{ 2\kappa T} L\right)\ge e^{-50 L^2 \kappa   T}.
\]
In addition, we observe that there exists some $c_{1,L}>0$ which depends on $\mu$, for which  
\[
\p(E_2 ) = \p\left(\sup_{[0,T]} |G|\leq L \right) \ge \p\left(\sup_{[0,T]} |F|\leq L \right) \ge e^{-c_{1,L}T},
\]
where the first inequality follows from Anderson's inequality (Lemma~\ref{lem:ander}) and the second one follows from the ball estimate of Lemma~\ref{lem: FFN ball}. From the latter lemma and from the fact that $\lim_{\ell\to\infty} \p\left( c |F(0)|\le \ell \right)=1$ we also deduce that $\lim_{L\to \infty}c_{1,L}=0$.

Similarly we obtain the existence of $c_{2,L}$ which satisfies $\lim_{L\to \infty}c_{2,L}=0$, for which
\[
\p(E_3 ) = \p\left(\sup_{[0,T]} |G'|\leq aL \right) \ge \p\left(\sup_{[0,T]} |F'|\leq aL \right) \ge e^{-c_{2,L}T}.
\]
Since $G, G'$ are almost surely continuous, we may apply the Gaussian correlation inequality (Lemma~\ref{lem:GCI}) to obtain that
$\p(E_2\cap E_3)\ge \p(E_2)\p(E_3).$ 
Setting $L=\kappa^{-1/3}$ and plugging the estimates we obtained into~\eqref{eq: lower} we get
\[
\p \left( \frac{N(T)}T \in \left[\frac{1-2\ep}{\pi} ,\frac{1+2\ep}{\pi} \right] \right) \ge 
e^{-(50L^{-1}+c_{1,L}+c_{2,L}) T}.
\]
To obtain the moreover part, apply this for arbitrarily small $\ep=\delta$, and notice that $L=L(\delta)$ is then arbitrarily large. 
\qed

\section{Upper bound on overcrowding probabilities}\label{sec:over}

In this section we prove the first part of Theorem \ref{thm_overcrowding}, which gives a Gaussian upper bound on certain overcrowding probabilities. Note that the second part of Theorem~\ref{thm_overcrowding} follows readily from Proposition~\ref{prop:lb}.

Throughout the remainder of the paper, we let $F$ be a fixed centred SGP with spectral measure $\mu$ supported on $[-A,A]$, which satisfies $\muac \not\equiv 0$.
Firstly we make the observation that, since $\mu$ is compactly supported, $F$ is almost surely real analytic, so that for all $t\in \re$ we have
\begin{align}\label{eq:real_analytic}
F(t) = \sum_{k=0}^\infty \frac{F^{(k)}(0)t^k}{k!}.
\end{align}
This formula naturally extends $F$ to an entire function on $\comp$. By a slight abuse of notation, we denote this extended function also by $F$. 

In order to apply the Jensen's formula to $F$, we require the following proposition.

\begin{Prop}\label{new_lemma}
$F$ is almost surely in Cartwright's class and of exponential type at most $A$.
\end{Prop}

Using this proposition we develop the following probabilistic estimate.

\begin{Prop} \label{prop:key} There exists $c>0$ such that for all $T>e^8$,
\begin{align*}
	\P\left(\forall a\in [0,T],\: r \in (0,T): \:\frac{1}{2\pi}	\int_{0}^{2\pi} \log |F(a+re^{i\theta})| d\theta \leq  \frac{A}{\pi} \cdot 2r + 6 \log T \right)\ge 1-e^{-cT^2}.
\end{align*}
\end{Prop}
The proof of both propositions is provided in Section~\ref{sec:analytic proof} below.
Equipped with them we are now ready to prove the theorem itself.

\subsection{Proof of the upper bound in Theorem \ref{thm_overcrowding}} \label{sec:proof_of_OC}
Let $T\ge e^8$.
Given $\ep\in \left[\frac 1 {\sqrt T}, e^{-3}\right]$, write 
\begin{equation}\label{eq: notes}
\Delta=\ep^2T, \:\: r=\ep T \text{  and  } n=\left\lceil{\frac{T+2r}{\Delta}}\right\rceil = \left\lceil \frac 1{\ep^2} +\frac 2 \ep\right\rceil.
\end{equation}
Let $x_1,\dots, x_n \in [0,T]$ be such that $$x_k\in I_k:=[(k-1)\Delta -r,k\Delta-r]$$ and assume that $F(x_k)\neq 0$ for all $k\in [n]$. By Jensen's formula (Theorem~\ref{Jensens}), for all $r>0$ we have
\begin{equation}\label{eq:Jen's app}
\sum_{k\in [n]}\int_{0}^{r} \frac{n_{F}(x_k,t)}{t} dt =   \sum_{k\in [n]} \frac 1{2\pi} \int_0^{2\pi} \log|F(r e^{i\theta}+x_k)| d\theta - \sum_{k\in [n]}\log|F(x_k)|.
\end{equation}
We proceed in three steps. Firstly, we relate the left-hand-side of \eqref{eq:Jen's app} to $N_{F}(T)$. Next, we use Proposition~\ref{prop:key} to replace the first term in the right-hand-side with the desired bound and show that with high probability there exists a good collection of points $x_1,\dots,x_n$ for which the second term of the right-hand-side is negligible. Finally, we combine these to obtain the desired upper bound.
\vskip 2mm
\noindent \textbf{Step 1: Relating the left-hand-side of \eqref{eq:Jen's app} and the number of real zeros.} 
Observe that for all $x\in [0,T]$, $t>0$ we have $$|\{k\in[n]:\ x\in [x_k-t,x_k+t]\}|\ge \left\lceil\frac{2t}{\Delta}\right\rceil-2.$$ In particular, this holds for values of $x$ which are real zeros of $F$.
Denoting the set of real zeros of $F$ by $\mathscr{Z}$, we therefore have 
\begin{align*}
\sum_{k\in [n]} n_F(x_k,t) \ge 
\sum_{k\in [n]}\sum_{w\in \mathscr{Z}}\mathbbm{1}_{(x_k - t, x_k +t)}(w) \ge
\left(\left\lceil\frac{2t}{\Delta}\right\rceil-2\right) N_{F}(T) 
\end{align*}
Integrating both sides against $\tfrac{1}{t}$ yields 
\begin{align}\label{eq: relation to zero count}
\sum_{k\in [n]} \int_{0}^{r} \frac{n_F(x_k,t)}{t}  dt &\ge 
N_{F}(T)\int_{\Delta}^{r} \frac{1}{t}\left(\left\lceil\frac{2t}{\Delta}\right\rceil-2\right)dt \notag
\\ & \ge
 2N_F(T)\lb\frac{r-\Delta}{\Delta}-\log\lb\frac{r}{\Delta}\rb\rb \notag \\ &=
 2N_F(T) \lb\frac 1 {\ep} -1 -\log \left(\frac 1 \ep\right)\rb \notag \\ 
& \geq 2 N_F(T) \lb \frac 1 \ep - 2\log \lb\frac{1}{\ep}\rb\rb,
\end{align}
where we used~\eqref{eq: notes} and the fact that $\ep<e$. 
\vskip 2mm
\noindent
\textbf{Step 2: Bounding the right-hand-side of~\eqref{eq:Jen's app}.} 
By Proposition~\ref{prop:key}, there exists $c_1>0$ such that the event
\[E_1:=\left\{\sum_{k\in [n]} \frac 1 {2\pi} \int_0^{2\pi} \log|F(r e^{i\theta}+x_k)| d\theta\le \frac{A}{\pi}\cdot 2r n+  6 n \log T\right\}\]
satisfies
\begin{equation}\label{eq: E1}
 \P(E_1)\ge 1-e^{-c_1 T^2}.
 \end{equation}
By Lemma~\ref{lem:smallballprob}, there exist $C,c_2>0$ such that if $\ep^2 T>1$ we have
\[
\p\left(\sup_{I_k} |F| \leq e^{-C\ep^2 T}\right)
=\p\left(\sup_{[0,\ep^2 T]} |F| \leq e^{-C\ep^2 T}\right)
\leq e^{-c_2 \ep^4 T^2}.
\]
Writing 
\[E_2:=\left\{\forall k\in [n]: \sup_{I_k} |F| \geq e^{-C\ep^2 T}\right\},\]
taking a union bound over $k\in[n]$, and using the fact that
$
n\le \frac 1{\ep^2}+\frac 2 \ep +1 \le \left( \sqrt{T} +1\right)^2 \le 2T$, we obtain 
\begin{equation}\label{eq: E2}
\p\left(E_2\right) \ge 1- n e^{-c_2 \ep^4 T^2}\ge 1- 2T e^{-c_2 \ep^4 T^2}.
\end{equation}

On the event $E_2$, we can select the points $\{x_k\}_{k\in[n]}$ which satisfy $|F(x_k)| \ge e^{-C\ep^2 T}$.
Thus, on the event $E_1\cap E_2$, we can bound the right-hand-side of~\eqref{eq:Jen's app} by
\begin{equation}\label{eq: step2}
 \sum_{k\in [n]} \frac1{2\pi}\int_0^{2\pi} \log|F(r e^{i\theta}+x_k)| d\theta - \sum_{k\in [n]}\log|F(x_k)|
 \le \frac {A}{\pi} \cdot 2rn +  6 n \log T + Cn \ep^2 T.
\end{equation}
\vskip 2mm
\noindent
\textbf{Step 3: Obtaining the probabilistic bound.} 
Plugging \eqref{eq: relation to zero count} and \eqref{eq: step2} into \eqref{eq:Jen's app} we obtain that, on the event $E_1\cap E_2$, we have
\[ 
2\lb \tfrac{1}{\ep} - 2\log(\tfrac{1}{\ep}) \rb N_F(T)\le \frac{A}{\pi}\cdot 2rn  +6n \log T +Cn\ep^2 T.
\]
Recalling~\eqref{eq: notes} and the fact that $n\ep^2 \le 1+3\ep\le 2$, we conclude that on $E_1\cap E_2$,
\begin{align*}
{(1-2\ep \log(\tfrac{1}{\ep}))} N_F(T) &\le \frac{A}{\pi} T(1+3\ep) +  3 (\tfrac 1 \ep +3)\log T + C\ep T\\
&
\le  \frac{A}{\pi} T  + D \ep T + D \tfrac 1 \ep \log T,
\end{align*}
where $D>0$ is some constant. 
In what follows we use $D$ to denote constants which may change from line to line.
Since $(1-2\ep \log(\tfrac{1}{\ep}))^{-1} \leq 1 +3\ep \log(\tfrac{1}{\ep})$ for $\ep<e^{-3}$, we get that on $E_1\cap E_2$,
\begin{align*}
N_F(T) &\le \frac{A}{\pi}T + D \ep \log\left( \tfrac 1 \ep\right)T +  {D\tfrac 1 \ep \log T}
\\ & \le \frac{A}{\pi}T + D \ep \log\left( \tfrac 1 \ep\right)T,
\end{align*}
{where in the last inequality we used that $\ep>\frac{1}{\sqrt{T}}$.}
By~\eqref{eq: E1} and~\eqref{eq: E2} we have $\p (E_1\cap E_2) \ge 1- e^{-c \ep^4 T^2}$ for some $c>0$.
Changing variables by defining $\delta = D\ep\log\left(\tfrac 1 \ep\right)$, we conclude that
\[
\P \left(N_F(T)\le \tfrac{A}{\pi }T+\delta T\right) \ge \P(E_1\cap E_2) \ge 1- \exp\left(-c \left(\delta/\log \delta \right)^4 T^2\right),
\]
where $\delta \in [d_1 \log T / \sqrt T,d_2]$, for some constants $d_1$, $d_2>0$.
\qed

\subsection{Proof of Propositions~\ref{new_lemma} and~\ref{prop:key}}\label{sec:analytic proof}
Here we establish Proposition~\ref{new_lemma} and use it together with tools from Sections \ref{subs: perlim: entire} and \ref{subs: perlim: gaussian}
to establish Proposition~\ref{prop:key}.

\begin{proof}[Proof of Proposition~\ref{new_lemma}]
 We note that the integral condition~\eqref{eq:defn_cartwright} holds by Lemma~\ref{lem:lineargrowth}, which ensures that, almost surely, $F$ grows at most linearly on $\R$.
It remains to show that the exponential type of $F$ is at most $A$.
Recall that $F$ has the power series expansion~\eqref{eq:real_analytic}:
		\begin{align*}
		F(z) = \sum_{k=0}^\infty \frac{F^{(k)}(0)}{k!} z^k, \quad \forall z\in \comp.  
		\end{align*}
Note that, for $k \in \nat$, 
 we have $F^{(k)}(0) \sim \mathcal{N}(0, {C_{2k}})$,
 where $C_{2k} := \int_{\re}|\lambda|^{2k} d\mu(\lambda)$. Using the fact that $\sprt(\mu)\subseteq [-A,A]$, we obtain $C_{2k} \leq A^{2k}$. An application of Lemma~\ref{lem: tail} yields for $k\ge 2$:
		\begin{align*}
		\p(|F^{(k)}(0)| \geq kA^k) \leq \p(|F^{(k)}(0)| \geq k \sqrt{C_{2k}}) \leq \exp(-k^2/2).
		\end{align*}
		Hence, almost surely, there exists $m \in \nat$ such that $|F^{(k)}(0)| \leq kA^k$ for every $k \geq m+1$. Denoting $C=\max_{k\le m} |F^{(k)}(0)/A^k|$, we obtain
		\begin{align*}
		|F(z)| &\leq \sum_{k=0}^{m} |F^{(k)}(0)| \frac{|z|^k}{k!} + \sum_{k=m+1}^{\infty} kA^k \frac{|z|^k}{k!}\\
		& \leq C \sum_{k=0}^{m} A^k \frac{|z|^k}{k!} + A|z|  \sum_{k=m}^{\infty} A^{k} \frac{|z|^{k}}{k!}\\
  &\leq (C+A|z|) \exp({A|z|}).
		\end{align*}
For any $\ep>0$, there exists $C_\ep>0$ such that $C+A|z| \leq C_\ep e^{\ep |z|}$ for all $z\in\comp$, which, in turn, implies $|F(z)|\le C_\ep e^{(A+\ep)|z|}$. Recalling 
the definitions in Section~\ref{sec: story up},
we deduce that $F$ is a.s. of exponential type at most $A$.
\end{proof}
\begin{proof}[Proof of Proposition~\ref{prop:key}]
With Proposition~\ref{new_lemma} at hand, apply Corollary~\ref{cor:analogue_of_pl} to get that 
for every $z=x+iy \in \comp$, we have
 \begin{align}\label{logmodf}
\log |F(z)| \leq \frac{|y|}{\pi} \int_{\re} \frac{\log^{+}|F(t)|}{|t-z|^2} dt + A |y|. 
\end{align}
Write 
\[
E_M = \{\forall t\in \re\,:\,|F(t)| \leq M(1+|t|)\}.
\]
By Lemma~\ref{lem:lineargrowth}, there exist $c>0$ such that 
\begin{equation}\label{eq: E_M}
    \P(E_M)\ge 1-e^{-c M^2},
\end{equation}  for sufficiently large values of $M$. 
Using the fact that $\frac y \pi \int_{\re} \frac 1 {|t-z|^2} dt =1$, we obtain from~\eqref{logmodf} that on the event $E_M$,
we have 
\begin{equation}\label{eq: logF UB}
\log |F(z)|\le \log M +  \frac{|y|}{\pi} \int_{\re} \frac{\log(1+|t|)}{|t-z|^2} dt+A|y|,  \quad \forall z\in \comp.
\end{equation}
For $y\ne 0$ and $|z|\le 2T$, we obtain
		\begin{align}\label{eq: int}
		  \frac{|y|}{\pi} \int_{\re} \frac{ \log(1+|t|)}{(t-x)^2 + y^2} dt &=  \frac{|y|}{\pi} \int_{\re} \frac{ \log(1+|x+s|)}{s^2 + y^2} ds\nonumber\\
		 &\leq  \frac{|y|}{\pi} \int_{\re} \frac{ \log(1+2T+|s|)}{s^2 + y^2} ds \nonumber
   \\ & \le 
   2\log (1+2T) + \frac{2|y|}{\pi} \int_{\re} \frac{\log^{+}|s|}{s^2 + y^2} ds \nonumber
   \\ & \le 2\log (1+2T) + 2\log^{+}|y| +  2 \le 4(2+\log T),
		\end{align}
where for the second inequality we use the fact that $\log(\alpha+\beta) \leq 2(\log \alpha + \log^{+} \beta)$ for $\alpha \geq 1$ and $\beta>0$, and for the third we use the bound  
\[
\frac{|y|}{\pi} \int_{\re} \frac{\log^{+}|s|}{s^2 + y^2} ds = \frac{1}{\pi} \int_{\re} \frac{\log^{+}|ty|}{1+t^2} dt = \log^{+}|y| +  \frac{1}{\pi} \int_{\re} \frac{\log^{+}|t|}{1+t^2} dt
\le  \log^{+}|y| +1. 
\]
Plugging~\eqref{eq: int} into~\eqref{eq: logF UB}, we conclude that on the event $E_M$,
for $|z|\le 2T$ with $\text{Im}(z)\ne 0$,
\begin{align}\label{eq:last_one}
		\log |F(z)| \leq A|y| + \log M + 4(2+\log T).
		\end{align}
 Thus on the event $E_M$, for all $a \in [0,T]$, $r \in (0,T)$, and $\theta \in (0,\pi) \cup (\pi,2\pi)$ we have
		\begin{align}\label{eq:last_two}
		\log |F(a + re^{i\theta})| \leq Ar |\sin \theta| +  \log M + 4(2+\log T).
		\end{align}
Set $M=T$. By integrating on $\theta \in (0,2\pi)$, we obtain that on the event $E_T$,
\[
 \forall a \in [0,T], \: r \in (0,T): \quad
 \frac 1 {2\pi}\int_0^{2\pi} \log |F(a+re^{i\theta})| \le \frac{2}{\pi} A r + 6\log T,
\]
as long as $T\ge e^{8}$. Recalling~\eqref{eq: E_M}, this completes the proof.
\end{proof}

\section{Upper bound on undercrowding probabilities}\label{sec:under}
This section consists of the proof Theorem \ref{thm_undercrowding}, by a reduction to Theorem \ref{thm_overcrowding}. Let $\mu$, $A$ and~$B$ be as in the statement of Theorem~\ref{thm_undercrowding}.
    We define a new measure $\nu$ as follows: for any  Borel measurable $S\subseteq \R$,
    \begin{align*}
	\tmu (S) := \mup(S+A) + \mum(S - A),
	\end{align*}
   where $\mup(S) = \mu(S \cap [0,\infty))$ and $\mum(S) = \mu(S \cap (-\infty,0)).$	
    Observe that $\tmu$ is a symmetric Borel probability measure on $\re$ supported on $[B-A, A-B]$, so that it is the spectral measure of a centered SGP on $\re$ which we denote by $G$.

	\begin{Claim}\label{claim_undercrowding} 
 The process $G$ has the representation
		\begin{align}\label{eqnn4}
		G(x) = \cos (Ax) F(x) + \sin(Ax) H(x),
		\end{align}
		where $H$ has the same distribution as $F$ (but is dependent of it).
		\end{Claim}
	
 First, we use this to prove Theorem \ref{thm_undercrowding} and then establish the claim itself. 
     Note that, by~\eqref{eqnn4}, 
	\begin{align*}
	G({k \pi}/{A}) = (-1)^k  F(k\pi/A), \quad \text{for any } k\in \mathbb{Z}.
	\end{align*}
	By continuity, we deduce that if $G$ has no zeros in the interval $[k_0\tfrac \pi A, (k_0+1)\tfrac \pi A]$ for some $k_0\in \Z$, then $F$ must have at least one zero in that interval. Since both $F$ and $G$ are non-constant SGPs, their zero set is almost surely disjoint from the lattice $(\pi/A)\mathbb{Z}$. We conclude that 
    \begin{equation}\label{eq:undc_402} 
        N_F(T)\ge \left\lfloor\frac{A T}{\pi}\right\rfloor - N_G(T).
    \end{equation}
	Since $G$ satisfies all the assumptions of Theorem \ref{thm_overcrowding}, we may use it to obtain the existence of $c,d_1,d_2 >0$ such that, for every $T>e^8$ and 
 $\ep \in \left[d_1{{\tfrac {\log T}{\sqrt T}}}, d_2\right]$,, we have
	\begin{align*}
	\p \left( N_G(T) \geq   \left\lceil \frac{A-B}{\pi}\,T\right\rceil + \ep T  \right) \le \exp(-c(\ep/\log \ep)^4 T^2).
	\end{align*}
This together with \eqref{eq:undc_402} yields the upper bound in Theorem~\ref{thm_undercrowding}. \qed
\begin{proof}[Proof of Claim \ref{claim_undercrowding}]
Let $\{f_n\}_{n=1}^{\infty}$ be an orthonormal basis for $\Lsym$ as per Lemma~\ref{lem:series_repn}. For every $n \in \nat$, consider the function $g_n$ defined as follows:
\begin{align*}
	g_n(t) := \begin{cases} 
f_n(t+A)\text{, if }t<0,\\
f_n(t-A)\text{, if }t\geq 0. 
\end{cases}
\end{align*}
Note that $\{g_n\}_{n=1}^{\infty}$ form an orthonormal basis for $\mathcal{L}^{2}_{\text{symm}}(\nu)$. Lemma~\ref{lem:series_repn} thus implies
\begin{equation}\label{eq: G series}
F \overset{d}{=} \sum_{n=1}^\infty \xi_n \widehat{f}_n,\quad \text{and} \quad G \overset{d}{=} \sum_{n=1}^\infty \xi_n \widehat{g}_n,
\end{equation}
where $\xi_n \overset{\text{i.i.d.}}\sim \cN (0,1)$, where
\[
\widehat{f}_n(x) = \int_{\re} f_n(\lm) e^{-i\lm x} d\mu (\lm), \quad \widehat{g}_n(x) = \int_{\re} g_n(\lm) e^{-i\lm x} d\tmu (\lm).
\]
Using the relation between $g_n$ and $f_n$, we compute:
\begin{align}\label{eq:decom}
	\notag\widehat{g}_n(x) &= \int_{\re} g_n(\lm) e^{-i\lm x} d\tmu (\lm) \\
	&\notag = \int_{B-A}^0 f_n(\lm+A) e^{-i\lm x} d\tmu(\lm) +  \int_{0}^{A-B} f_n(\lm-A) e^{-i\lm x} d\tmu(\lm) \\
	& \notag= e^{iAx} \int_{B}^{A} f_n(\tau) e^{-i\tau x} d\mu(\tau) + e^{-iAx} \int_{-A}^{-B} f_n(\tau) e^{-i\tau x} d\mu(\tau ) \\ 
 &\notag = (\cos(Ax) + i\sin(Ax) )  \int_{B}^{A} f_n(\tau) e^{-i\tau x} d\mu(\tau) 
 + (\cos(Ax) - i\sin(Ax)) \int_{-A}^{-B} f_n(\tau) e^{-i\tau x} d\mu(\tau )\\
	& = \cos(Ax) \cdot \widehat{f}_n(x) + \sin(Ax) \cdot h_n(x)
	\end{align}
	where 
	\begin{align*} 
 h_n &=i\lb \int_{B}^A f_n(\tau) e^{-i\tau x} d\mu(\tau) 
 - \int_{-A}^{-B} f_n(\tau) e^{-i\tau x} d\mu(\tau ) \rb \\ 
 &= i\int_\R \left(f_n \mathbbm{1}_{[0,\infty)}(\tau) - f_n\mathbbm{1}_{(-\infty,0)}(\tau)\right)e^{-i \tau x} d\mu(\tau).
 \end{align*}
	Note that
 $\{i f_n \mathbbm{1}_{[0,\infty)} -i f_n\mathbbm{1}_{(-\infty,0)}\}$ also forms a basis to $\Lsym$,
 so that the random series $H := \sum_{n=1}^{\infty} \xi_n h_n$ is a SGP with spectral measure $\mu$ (by means of Lemma~\ref{lem:series_repn}). 
Plugging~\eqref{eq:decom} into~\eqref{eq: G series}, we obtain 
	\begin{align*}
	G(x) &= \sum_{n=1}^{\infty} \xi_n \widehat{g}_n(x) = \cos(Ax) \sum_{n=1}^{\infty} \xi_n \widehat{f}_n(x) + \sin (Ax) \sum_{n=1}^{\infty} \xi_n h_n(x),\\
	& = \cos(Ax) F(x) + \sin(Ax) H(x),
	\end{align*}
which completes the proof of Claim \ref{claim_undercrowding}.
\end{proof}
\section*{Acknowledgements}

We are very grateful to Mikhail Sodin, Alon Nishry and  Manjunath Krishnapur for their encouragement and for many stimulating discussions on the subject of this paper. 

	\bibliographystyle{plain}
	\bibliography{reference_oc}
	\end{document}